\documentclass{agtart_a}
\pdfoutput=1
\usepackage[all]{xy}

\title{A rational splitting of a based mapping space}
\author{Katsuhiko Kuribayashi}
\givenname{Katsuhiko}
\surname{Kuribayashi}
\address{Department of Mathematical Sciences\\
Faculty of Science\\
Shinshu University\\\newline
Matsumoto\\
Nagano 390-8621\\
Japan}
\email{kuri@math.shinshu-u.ac.jp}

\author{Toshihiro Yamaguchi}
\givenname{Toshihiro}
\surname{Yamaguchi}
\address{Department of Mathematics Education\\
Faculty of Education\\
Kochi University\\\newline
Kochi 780-8520\\
Japan}
\email{tyamag@cc.kochi-u.ac.jp}



\received{19 July 2005}
\revised{14 February 2006}
\accepted{14 February 2006}
\proposed{}
\seconded{}
\publishedonline{7 March 2006}
\published{7 March 2006}

\volumenumber{6}
\issuenumber{}
\publicationyear{2006}
\papernumber{10}
\startpage{309}
\endpage{327}

\doi{}
\MR{}
\Zbl{}

\arxivreference{}  
\arxivpassword{}

\let\xysavmatrix\xymatrix
\def\xymatrix{\disablesubscriptcorrection\xysavmatrix}
\AtBeginDocument{\let\bar\wbar\let\tilde\wtilde}
\newcommand{\we}{\smash{\rlap{\kern 6pt 
\raise 4pt\hbox{\footnotesize $\sim$}}}\longrightarrow}
\numberwithin{equation}{section}

\makeatletter
\def\cnewtheorem#1[#2]#3{\newtheorem{#1}{#3}[section]
\expandafter\let\csname c@#1\endcsname\c@thm}

\newtheorem{thm}{Theorem}[section]
\cnewtheorem{prop}[thm]{Proposition}
\cnewtheorem{cor}[thm]{Corollary}
\cnewtheorem{lem}[thm]{Lemma}
\theoremstyle{definition}
\cnewtheorem{defn}[thm]{Definition}
\cnewtheorem{notation}[thm]{Notation}
\theoremstyle{remark}
\cnewtheorem{rem}[thm]{Remark}
\cnewtheorem{exa}[thm]{Example}

\makeatother  

\newcommand{\F}{{\mathcal F}}

\newcommand{\e}{\varepsilon}

\newcommand{\hq}[1]{%
 H^*({#1}; {\mathbb Q})}

\newcommand{\mapright}[1]{%
 \smash{\mathop{%
  \hbox to 1cm{\rightarrowfill}}\limits_{#1}}}
\newcommand{\maprightd}[2]{%
 \smash{\mathop{%
  \hbox to 1.2cm{\rightarrowfill}}\limits^{#1}\limits_{#2}}}
\newcommand{\mapleft}[1]{%
 \smash{\mathop{%
  \hbox to 1cm{\leftarrowfill}}\limits_{#1}}}
\newcommand{\mapleftu}[1]{%
 \smash{\mathop{%
  \hbox to 1cm{\leftarrowfill}}\limits^{#1}}}
\newcommand{\mapleftud}[2]{%
 \smash{\mathop{%
  \hbox to 1.2cm{\leftarrowfill}}\limits^{#1}\limits_{#2}}}

\newcommand{\maprightu}[1]{%
 \smash{\mathop{%
  \hbox to 1cm{\rightarrowfill}}\limits^{#1}}}
\newcommand{\maprightud}[2]{%
 \smash{\mathop{%
  \hbox to 1cm{\rightarrowfill}}\limits^{#1}_{#2}}}

\begin{document}

\begin{asciiabstract}
Let F_*(X, Y) be the space of base-point-preserving maps from a
connected finite CW complex X to a connected space Y.  Consider a CW
complex of the form X cup_{alpha}e^{k+1} and a space Y whose
connectivity exceeds the dimension of the adjunction space.  Using a
Quillen-Sullivan mixed type model for a based mapping space, we prove
that, if the bracket length of the attaching map alpha: S^k --> X is
greater than the Whitehead length WL(Y) of Y, then F_*(X
cup_{alpha}e^{k+1}, Y) has the rational homotopy type of the product
space F_*(X, Y) times Omega^{k+1}Y.  This result yields that if the
bracket lengths of all the attaching maps constructing a finite CW
complex X are greater than WL(Y) and the connectivity of Y is greater
than or equal to dim X, then the mapping space F_*(X, Y) can be
decomposed rationally as the product of iterated loop spaces.
\end{asciiabstract}

\begin{htmlabstract}
Let F<sub>*</sub>(X,Y) be the space of base-point-preserving maps from a
connected finite CW complex X to a connected space Y.  Consider a
CW complex of the form X&cup;<sub>&alpha;</sub>e<sup>k+1</sup> and a
space Y whose
connectivity exceeds the dimension of the adjunction space.  Using a
Quillen&ndash;Sullivan mixed type model for a based mapping space, we prove
that, if the <i>bracket length</i> of the attaching map
&alpha;:S<sup>k</sup>&rarr;X is greater than the Whitehead length WL(Y) of Y,
then F<sub>*</sub>(X&cup;<sub>&alpha;</sub>e<sup>k+1</sup>,Y) has the
rational homotopy type
of the product space F<sub>*</sub>(X,Y)&times;&Omega;<sup>k+1</sup>Y.
This result
yields that if the bracket lengths of all the attaching maps
constructing a finite CW complex X are greater than
WL(Y) and the connectivity of Y is greater than or equal
to dim X, then the mapping space F<sub>*</sub>(X,Y) can be
decomposed rationally as the product of iterated loop spaces.
\end{htmlabstract}

\begin{abstract}
Let $\mathcal{F}_*(X, Y)$ be the space of base-point-preserving maps from a
connected finite CW complex $X$ to a connected space $Y$.  Consider a
CW complex of the form $X\cup_{\alpha}e^{k+1}$ and a space $Y$ whose
connectivity exceeds the dimension of the adjunction space.  Using a
Quillen--Sullivan mixed type model for a based mapping space, we prove
that, if the {\it bracket length} of the attaching map $\alpha : S^k
\to X$ is greater than the Whitehead length $\mathrm{WL}(Y)$ of $Y$,
then $\mathcal{F}_*(X\cup_{\alpha}e^{k+1}, Y)$ has the rational homotopy type
of the product space $\mathcal{F}_*(X, Y)\times \Omega^{k+1}Y$.  This result
yields that if the bracket lengths of all the attaching maps
constructing a finite CW complex $X$ are greater than
$\mathrm{WL}(Y)$ and the connectivity of $Y$ is greater than or equal
to $\mathrm{dim} X$, then the mapping space $\mathcal{F}_*(X, Y)$ can be
decomposed rationally as the product of iterated loop spaces.
\end{abstract}

\maketitle

\section{Introduction}\label{sec:1}
Let $X$ be a connected finite CW complex with basepoint 
and $X\cup_{\alpha}e^{k+1}$ the adjunction space obtained 
by attaching the cell $e^{k+1}$ to $X$ along 
a cellular map $\alpha \co  S^k \to X$. 
Let $\F_*(X, Y)$ denote the space of base-point-preserving maps 
from  $X$ to a connected space $Y$ with basepoint. The cofibre sequence
$X \stackrel{i}{\to} X\cup_{\alpha}e^{k+1}  \stackrel{j}{\to} S^{k+1}$
gives rise to  the fibration
$$
\Omega^{k+1}Y = \F_*(S^{k+1}, Y) \stackrel{j^\sharp}{\to}
\F_*(X\cup_{\alpha}e^{k+1}, Y)  \stackrel{i^\sharp}{\to}
\F_*(X, Y).
$$
The aim of this article is to consider when the above fibration splits 
after localization at zero. 
Roughly speaking,  our main theorem described below asserts that such a
splitting is possible if a number which expresses complexity of the
attaching map
$\alpha \co  S^k \to X$ is greater than the nilpotency of the rational
homotopy Lie algebra of $Y$. 
In order to state the theorem more precisely, we first introduce 
the number associated with a map $\alpha \co  S^k \to X$.
Let $L$ be a graded Lie algebra. We define  a subspace $[L, L]^{(l)}$
of $L$ by $[L, L]^{(l)} = [L, [L, [ ..., [L, L] ... ]]$ ($l$--times)
and  $[L, L]^{(0)}= L$, where
$[ \ , \ ]$ denotes the Lie bracket of $L$.
Observe that $[L, L]^{(l+1)}$ is a subspace of $[L, L]^{(l)}$.

\begin{defn}
Let $X$ be a simply-connected space. 
The {\it bracket length} of a map $\alpha \co  S^k \to X$,
written $\text{bl}(\alpha)$, is
the greatest integer $n$ such that
the class of the adjoint map $\text{ad}(\alpha) \co  S^{k-1} \to \Omega X$ to 
$\alpha$ is in $[L_X, L_X]^{(n)}$, where $L_X$ 
denotes the homotopy Lie algebra $\pi_*(\Omega X)\otimes \Q$.
If the map $\text{ad}(\alpha)$ is in  $[L_X, L_X]^{(n)}$ for any $n$, then 
$\text{bl}(\alpha)=\infty$. 
\end{defn}

Recall the {\it Whitehead length} $\text{WL}(Y)$
of $Y$ which is the greatest integer $n$ such that
$[L_Y, L_Y]^{(n)}\neq 0$ (see for example Berstein and Ganea \cite{B-G}).

In what follows, we assume that
a space is based and its rational cohomology is locally
finite. The connectivity of a space $Y$ may be denoted by $\text{Conn}(Y)$.
For a nilpotent space $X$, we denote by $X_\Q$ the $\Q$--localization of $X$. 
Our main theorem can be  stated as follows: 

\begin{thm}
 \label{thm:main}
 Let $\alpha \co  S^k \to X$ be a cellular map from the
$k$--dimensional sphere to a simply-connected finite CW complex $X$,
where $k > 0$.
Let $Y$ be a space such that $\text{\em Conn}(Y)\geq 
\text{\em max}\{k+1, \dim X\}$.
If $\text{\em bl}(\alpha) > \text{\em WL}(Y)$,
then the fibration
\begin{equation}\label{eq:1.1}
\Omega^{k+1}Y=\F_*(S^{k+1}, Y) \stackrel{j^\sharp}{\to}
\F_*(X\cup_{\alpha}e^{k+1}, Y)  \stackrel{i^\sharp}{\to}
\F_*(X, Y)
\end{equation}
is rationally trivial; that is, there is a
 homotopy equivalence 
$$
\F_*(X\cup_{\alpha}e^{k+1}, Y)_{\Q}
\stackrel{\simeq}{\to}  (\F_*(X, Y) \times \Omega^{k+1}Y)_{\Q}
$$
which covers the identity map on $\F_*(X, Y)_{\Q}$.
\end{thm}

Suppose that $Y$ is a connected nilpotent space and 
$X$ is a finite CW complex. Then $\F_*(X, Y)$ is a connected nilpotent
space (Hilton, Mislin and Roitberg \cite[Theorem 2.5, Chapter II]{H-M-R}). 
Moreover, $\F_*(X, Y)_\Q$ is homotopy equivalent to $\F_*(X, Y_\Q)$ 
\cite[Theorem 3.11, Chapter II]{H-M-R}. 

Suppose that $\alpha \co  S^k \to X$ is homotopic to the constant map. 
Then it is evident that  
$\F_*(X\cup_{\alpha}e^{k+1}, Y)_\Q \simeq (\F_*(X, Y) \times
\Omega^{k+1}Y)_\Q$. 
In this case, the bracket length of $\alpha$ is infinity. 
Thus we can regard that \fullref{thm:main} explains such
decomposition phenomena of mapping spaces more precisely from the rational
homotopy theory point of view.

As an immediate corollary, we have the following 
result on rational decomposition of a mapping space.  

\begin{thm}
 \label{thm:decomposition}
Let $X$ be a simply-connected finite CW complex and $Y$ a space such
 that $\text{\em Conn}(Y)\geq \dim X$. Suppose that the bracket length of each
 attaching map which constructs $X$ is greater than $\text{\em WL}(Y)$. Then 
$\F_*(X,Y)$ is rationally homotopy equivalent to the product space 
$\times_k(\Omega^k Y)^{n_k}$, where $n_k$ denotes  the number of the $k$--cells
 of $X$. In particular, $\F_*(X,Y)_\Q$ is a Hopf space.  
\end{thm}

In fact, by looking at the attaching maps with higher dimension in order  
and by applying \fullref{thm:main} repeatedly, we have the result.  

As an example, we give a mapping space $\F_*(X, Y)$ which admits the 
decomposition described in \fullref{thm:decomposition}. 
Construct a CW complex $X_n$ ($n \geq 0$) inductively 
as follows: Let $X_0$ be the $m_0$--sphere $S^{m_0}$, where $m_0 \geq 2$.   
Suppose that $X_i$ is defined. We fix $k$ integers 
$m(i)_j$ ($1\leq j \leq k$) greater than $1$. Moreover we choose an
element $\alpha_i \in \pi_{\deg \alpha_i}(X_i)$ and 
the generators $\iota_{m(i)_j} \in \pi_{m(i)_j}(S^{m(i)_j})$  
($1\leq j \leq k$).  
Define a CW complex $X_{i+1}$ by 
$$ 
X_{i+1} = (X_i\vee S^{m(i)_1}\vee \cdots \vee S^{m(i)_{k}})
\cup_{[\alpha_i, [\iota_{m(i)_1}\cdots 
 [\iota_{m(i)_{k-1}}, \iota_{m(i)_{k}}]\cdots]]}e^{l_i},    
$$
where $l_i = \deg \alpha_i + m(i)_1 + \cdots + m(i)_k -k + 1$. 
It follows that the bracket length of 
each attaching map is greater than or equal to $k$.  Let $Y$ be a space
which satisfies the condition that $k >  \text{WL}(Y)$ and 
$\dim X_n \leq \text{Conn}(Y)$. Then \fullref{thm:decomposition} 
enables us to conclude that
$$ 
\F_*(X_n, Y)\simeq_\Q \times_{i=0}^{n-1}(\Omega^{l_i}Y\times
\Omega^{m(i)_1}Y \times \cdots \times \Omega^{m(i)_k}Y) 
\times \Omega^{m_0}Y.    
$$
We here describe an application of \fullref{thm:decomposition}.  
\begin{cor}
 \label{cor:1} Let $X$ and $Y$ be the spaces which 
satisfy the conditions in Theorem 1.3. Then, for any space $Z$, 
there exist bijections of sets 
\begin{multline*}  
[ Z \wedge X, Y_\Q]_*\cong [Z, \F_*(X, Y_\Q)]_* \cong 
[Z, \times_k(\Omega^k Y)^{n_k}_\Q]_* \\ \cong \bigoplus_{m, k \geq 0, 
\pi_m(Y)\otimes\Q\neq 0}\hspace{-0.9cm}\widetilde{H}^{m-k}(Z; \Q)^{\oplus 
n_k},
\end{multline*}
where $n_k$ denotes  the number of the $k$--cells
 of $X$.
\end{cor}

We emphasize that {\it a Quillen--Sullivan mixed type
model} for a based mapping space, which is constructed out of
a model for a free mapping space due to
Brown and Szczarba \cite{B-S} (see \fullref{sec:2}), plays a crucial role in
proving
\fullref{thm:main}. 

The paper is organized as follows:
In \fullref{sec:2}, we recall a Sullivan model for a mapping space
constructed by Brown and Szczarba. The mixed type
model mentioned above is  described in this section. 
Moreover, we introduce a  numerical invariant $d_1\text{--depth}(Y)$,
which is called the $d_1\text{--depth}$ for a simply-connected space $Y$, 
using a filtration defined by the quadratic part of the differential 
of the minimal model for $Y$. 
This invariant is equal to the Whitehead length of $Y$. 
\fullref{sec:3} is devoted to proving \fullref{thm:main}. 
In the appendix (\fullref{sec:4}), we prove that $d_1\text{--depth}(Y) = \text{WL}(Y)$.  
It seems that the important equality is well known.
However, we could not find until recently 
any reference in which the equality has been 
proved explicitly. Kaji \cite{Kaji} has also 
proved it by looking at the nilpotency 
of the loop space $\Omega Y$. We wish to
stress that our proof of the equality 
in the appendix contains also a careful consideration 
on the filtration which defines the $d_1\text{--depth}$.

We end this section by fixing
some notations and terminology for this article.
A graded algebra $A$ is defined over the rational field $\Q$ and is
locally finite in the sense that
each vector space $A^i$ is finite dimensional.
Moreover it is assumed that an graded algebra $A$
is connected; that is, $A^0=\Q$ and $A^i =0$ for $i < 0$.
We denote by $\Q\{x_i\}$ the vector space with a basis $\{x_i\}$.
The free algebra generated by a graded vector space $V$ is denoted by
$\wedge V$ or $\Q[V]$. For an algebra $A$ and its dual coalgebra
$C$, we define $A^+$ and $C^+$ by $A^+=\oplus_{i>0}A^i$ and
$C^+=\oplus_{i<0}C_i$, respectively. 
Let $(B, d_B)$ be a differential graded algebra (DGA). We call 
a DGA $(B\otimes \wedge V, d)$ is a relative Sullivan algebra over 
$(B, d_B)$ if $d|_{B} = d_B$ and there exists an increasing filtration
$\{V(k)\}_{k\geq 0}$ such
that $V = \cup_kV(k)$ and $d(V(k))\subset B\otimes \wedge V(k-1)$.

\section{A Quillen--Sullivan mixed type model for a mapping space}\label{sec:2}

Let $(B, d_B)$ be a DGA and 
$(\wedge V, d)$ a minimal DGA; that is, $dv$ is decomposable for any 
$v \in V$.  
Let $B_*$ denote the differential graded coalgebra defined by
$B_q=\mbox{Hom}(B^{-q}, \Q)$
for $q\leq 0$ together with the coproduct $D$ and the differential 
$d_{B*}$, which are dual to the multiplication of $B$ and to the differential
$d_B$, respectively.
Let $I$ be the ideal of the free algebra $\Q[\wedge V \otimes B_*]$
generated by $1\otimes 1 -1$ and all elements of the form
$$
a_1a_2\otimes \beta_* -
\sum_i(-1)^{|a_2||\beta_i'|}(a_1\otimes \beta_{i*}')(a_2\otimes \beta_{i*}''),
$$
where $a_1, a_2 \in \wedge V$, $\beta_* \in B_*$ and $D(\beta_*)
= \sum_i\beta'_{i*}\otimes  \beta''_{i*}$. Observe that $\Q[\wedge V\otimes
B_*]$ is a DGA with the differential
$d := d_A\otimes 1\pm 1\otimes d_{B *}$.
The result of Brown and Szczarba \cite[Theorem 3.3]{B-S} yields that
$(d_A\otimes 1\pm 1\otimes d_{B *})(I)\subset I$.
Moreover it follows from \cite[Theorem 3.5]{B-S} that the composition map
 $$
 \rho \co  \Q[V\otimes B_*] \hookrightarrow \Q[\wedge V \otimes B_*] \to
 \Q[\wedge V \otimes B_*]/I
 $$
 is an isomorphism of graded algebras.
Thus we define a differential $\delta$ on
$\Q[V\otimes B_*]$
by $\rho^{-1}\tilde{d}\rho$, where $\tilde{d}$ is the differential on
$\Q[\wedge V \otimes B_*]/I$ induced by $d$.
The differential $\delta$ is described explicitly as follows:
For an element $v \in V$ and a cycle $\beta_* \in B_*$,
if  $d(v) = v_1\cdots v_m$ with $v_i \in V$, 
then
\begin{equation}\label{eq:2.1}
\begin{array}{lcl}
\delta(v\otimes (\beta_*)) &=&\sum_j v_1\cdots
v_m\cdot \beta_{j_1*}\otimes \cdots \otimes \beta_{j_m*} \\
&=&\sum_j(-1)^{\e(v_1,...,v_m,\beta_{j_1*},...,\beta_{j_m*})}v_1\otimes
\beta_{j_1*} \cdots v_m\otimes \beta_{j_m*}
\end{array}
\end{equation}
where
$D^{(m-1)}(\beta_*) =
\sum_j\beta_{j_1*}\otimes
\cdots \otimes \beta_{j_m*}$
with the iterated coproduct $D^{(m-1)}$ and the integer 
$(-1)^{\e(v_1,..,v_m,\beta_{j_1*},..,\beta_{j_m*})}$ is defined by the
formula 
$$
(-1)^{\e(v_1,..,v_m,\beta_{j_1*},..,\beta_{j_m*})}v_1\beta_{j_1}
\cdots v_m\beta_{j_m}=
v_1\cdots v_m \beta_{j_1}\cdots \beta_{j_m}
$$ 
in the graded algebra
$(\wedge V) \otimes B$ using elements $\beta_{j_s}$ ($a\leq s \leq m $)
with $\deg \beta_{j_s} = -\deg \beta_{j_s*}$.

We denote by $A_{PL}(X)$ the DGA of the polynomial 
differential forms on a space $X$.
Let $X$ be a connected finite CW complex
and $Y$ a connected space with $\dim X \leq \text{Conn}(Y)$.
We take a quasi-isomorphism $(B, d_B) \to A_{PL}(X)$ and a minimal model
$(\wedge V, d)$ for $Y$.
By applying the construction mentioned above, we obtain a
DGA of the form $(\Q[V\otimes B_*], \delta)$, which gives an algebraic model
(not minimal in general) for $\F(X, Y)$ the space of {\it free} maps
from $X$ to $Y$ \cite{B-S}. In fact, there exists a quasi-isomorphism
which connects $A_{PL}(\F(X, Y))$ with the DGA $(\Q[V\otimes B_*], \delta)$. 
Moreover, the realization of $(\Q[V\otimes B_*], \delta)$ is homotopy
equivalent to $\F(X, Y_\Q)$ \cite[Theorem 1.3]{B-S} 
and hence to $\F(X, Y)_\Q$.  
The result of the first author \cite[Proposition 5.3]{K}
asserts that $(\Q[V\otimes B_*], \delta)$ is a relative Sullivan algebra
with the base $\Q[V]$.   
Observe that $(\Q[V\otimes B_*], \delta)$ itself 
is a Sullivan algebra \cite[Reamrk 5.4]{K}.  
Moreover the model for $\F(X, Y)$ leads to that 
for the based mapping space $\F_*(X, Y)$.

\begin{thm}{\rm\cite[Theorem  4.3]{K}}\label{thm:basedmaps}\qua
There exist a quasi-isomorphism from a Sullivan algebra of the form
$(\Q[V\otimes B_*]/(\Q[V]^+),\overline{\delta})= 
(\Q\otimes_{\Q[V]} \Q[V\otimes B_*], 1\otimes \delta)$  
to $A_{PL}(\F_*(X, Y))$. Here $({\mathbb Q}[V]^+)$ is the ideal of 
$\Q[V\otimes B_*]$ generated by ${\mathbb Q}[V]^+$. 
%
\end{thm}

From the explicit form \ref{eq:2.1}
of the differential $\delta$, we can deduce the following lemma. The
proof is left to the reader. 

\begin{lem}
 \label{lem:1} 
Suppose that, for an element $v\otimes \beta_* \in V\otimes B_*^+$, 
 $dv$ is in $\wedge^{\geq m}V$ and
$\overline{D}^{m-1}(\beta_*) = 0$, where $\overline{D}^{m-1} \co  B_*^+ \to
(B_*^+)^{\otimes m}$
denotes the {\em(}$m-1${\em)} fold reduced coproduct. 
Then $\overline{\delta}(v\otimes \beta_*)=0$. 
In particular, $\overline{\delta}(v\otimes \beta_*)=0$
if $\beta_* \in B_*$ is a primitive element.
\end{lem}

We here recall, from F{\'e}lix--Halperin--Thomas \cite[Section
 22]{F-H-T}, Quillen's functor $C_*( \ )$ from the category of
 connected differential graded Lie algebras (DGL's) to the category of
 simply-connected cocommutative differential graded coalgebras
 (DGC's).  Let $(L, d_L)$ be a DGL and $\wedge (sL)$ be the
 primitively generated coalgebra over the vector space $sL$. We define
 the differentials $d_v$ and $d_h$ on $\wedge (sL)$ by
$$
d_v(sx_1\wedge \cdots \wedge sx_k)=-\sum_{i=1}^k(-1)^{n_i}
sx_1\wedge \cdots \wedge sd_Lx_i \wedge \cdots \wedge sx_k
$$
and
$$
d_h(sx_1\wedge \cdots \wedge sx_k)=\sum_{1\leq i < j\leq k}
(-1)^{|sx_i|+n_{ij}}s[x_i, x_j]\wedge sx_1 \cdots \widehat{sx_i} \cdots
\widehat{sx_j} \cdots \wedge sx_k,
$$
respectively. Here $n_i =\sum_{j<i}|sx_j|$ and
$sx_1 \wedge  \cdots \wedge sx_k =(-1)^{n_{ij}}
sx_i \wedge sx_j \wedge sx_1 \wedge \cdots \widehat{sx_i} \cdots
\widehat{sx_j} \cdots \wedge sx_k$.
We see that $C_*(L, d_L)=(\wedge (sL), d_v+d_h)$ is a DGC.
To simplify, we may write  $C_*(L)$ for  $C_*(L, d_L)$.  
By using the above DGC, we can construct a more explicit model for
a mapping space.
Let $(L, d_L)$ be a Lie model for a space $X$; that is, there exists a
quasi-isomorphism 
$C^*(L, d_L)=\text{dual} \ C_*(L, d_L) \stackrel{\simeq}{\to} 
A_{PL}(X)$.
We choose a minimal model $(\wedge V, d)$ for $Y$.
Then \fullref{thm:basedmaps} implies that 
the Sullivan algebra of the form 
$({\mathbb Q}[V\otimes C_*(L, d_L)]/({\mathbb
Q}[V]^+),\overline{\delta}) =
({\mathbb Q}[V\otimes C_*(L, d_L)^+], \overline{\delta})$
is a model for  the mapping space  $\F_*(X, Y)$.
This model, which is called
{\it a Quillen--Sullivan mixed type model} for the  based mapping space,
is an important ingredient for the proof of 
\fullref{thm:main}. 

\begin{rem}
The Sullivan algebra of the form
$({\mathbb Q}[V\otimes C_*(L, d_L)], \delta)$ is regarded as a
mixed type model for the free mapping space $\F(X, Y)$.
\end{rem}

We close this section by introducing
a  numerical invariant which is called the {\it $d_1$--depth}
of a given space. We use the invariant to prove \fullref{thm:main}.

Let $(\wedge V, d)$ be a minimal model for a simply-connected space
$Y$. Then the differential $d$
is decomposed uniquely as $d= d_1 + d_2 + \cdots$
in which  $d_i$ is a derivation raising the wordlength by $i$.
We call $d_1$  the quadratic part of $d$. 
We define a subspace $V_0$ of $V$
by $V_0 = \{v\in V \ | \  d_1(v) =0 \}$ and put $V_{-1}=0$.
Moreover, define a subspace $V_i$ inductively
by $V_i= \{v\in V \ | \ d_1(v)\in \wedge V_{i-1}\}$.
It is readily seen that $V_{k-1}\subset V_k$ and that if $V_{l}=V_{l-1}$,
then $V_k=V_{k+1}$ for $k\geq l$.

\begin{defn}
The {\it $d_1$--depth} of $Y$, denoted $d_1\text{--depth}(Y)$,
is the greatest integer $k$
such that $V_{k-1}$ is a proper subspace of $V_k$.
\end{defn}

It suffices to prove \fullref{thm:main} by assuming that
$\text{bl}(\alpha) > d_1\text{--depth}(Y)$ instead of the sufficient
condition  $\text{bl}(\alpha) > \text{WL}(Y)$.
The following theorem guarantees that the replacement is valid.

\begin{thm}
\label{thm:d_1}
 Let $Y$ be a simply-connected space.
Then $d_1\text{\em --depth}(Y)=\text{\em WL}(Y)$.
\end{thm}

\begin{proof}
See the appendix.
\end{proof}

Since the Whitehead length is a numerical topological invariant
in the category of the rational spaces, it follows that
the $d_1$--depth of $Y$
does not depend on the choice of minimal models for $Y$ and is also a
topological invariant.

\section[A minimal model and  Proof of \ref{thm:main}]{A minimal model and  Proof of \fullref{thm:main}}\label{sec:3}

Before proving \fullref{thm:main}, we recall from \cite{B-S}
a result concerning construction of a minimal model for a mapping
space.
Though the construction is for a free mapping space,
it is applicable to the model $(\Q[V\otimes
B_*]/(\Q[V]^+),\overline{\delta})$
for a based mapping space $\F_*(X, Y)$
which is described in \fullref{thm:basedmaps}.
With the notation in \fullref{sec:2}, we  write
$\Q[V\otimes B_*]/(\Q[V]^+)=\Q[V\otimes B_*^+]$. Let $\{a_k, b_k,  c_j\}$
be a basis for $B_*^+$ such that $d_{B_*^+}(a_k)=b_k$ and
$d_{B_*^+}(c_j)=0$. Choose a basis $\{v_i\}$ for $V$ so that
$|v_i|\leq |v_{i+1}|$ and $dv_{i+1}\in \Q[V_i]$, where
$V_i$ is the subspace spanned by the elements $v_1 ,..., v_i$.
The result \cite[Lemma 5.1]{B-S} states that there exist free algebra
generators $w_{ij}$, $u_{ik}$ and $v_{ik}$ such that
\begin{align}
&w_{ij}= v_i \otimes c_j + x_{ij},\text{ where }x_{ij}\in
\Q[V_{i-1}\otimes B_*^+],\label{eq:3.1}\\
&\overline{\delta}w_{ij}\text{ is decomposable and
in }\Q[\{w_{sl} ; s < i\}],\label{eq:3.2}\\
&u_{ik}= v_i\otimes a_k\text{ and }\overline{\delta}u_{ik}=v_{ik}.\label{eq:3.3}
\end{align}
Thus we have a decomposition
$\Q[V\otimes B_*^+]=\Q[w_{ij}]\otimes\Q[u_{ik}, v_{ik}]$ of a
differential graded algebra.
Since  $\Q[u_{ik}, v_{ik}]$ is contractible, it follows that the
inclusion $(\Q[w_{ij}], \overline{\delta})\to
(\Q[V\otimes B_*^+],\overline{\delta})$ is a quasi-isomorphism.
In consequence, we get a minimal model of the form
$(\Q[w_{ij}], \overline{\delta})$ for the mapping space $\F_*(X, Y)$.
Observe that the vector space generated by the elements
$w_{ij}$ is isomorphic to the reduced homology $H_*(B_*)^+$ as a vector
space.

We rely on the following result to construct a minimal model for the
mapping space $\F_*(X, Y)$ from 
the Sullivan algebra 
$({\mathbb Q}[V\otimes C_*(L, d_L)^+], \overline{\delta})$ in \fullref{sec:2}.

\begin{lem}{\rm \cite[Proposition 22.8]{F-H-T}}\label{lem:2}\qua
For a DGL of  the form $(\mathbb{L}_{W}, d_L)$,
let $\rho_1 \co  C_*(\mathbb{L}_{W})$ $=\wedge s\mathbb{L}_{W}\to
s\mathbb{L}_{W}\oplus \Q$ and $\rho_2 :
s\mathbb{L}_{W}\oplus \Q \to sW\oplus \Q$ be the maps obtained by
annihilating the factors $\wedge^{\geq 2}s\mathbb{L}_{W}$ and
$s(\mathbb{L}_{W}^{\geq 2})$, respectively. Then
the composition map
$\rho_2\circ\rho_1 \co  C_*(\mathbb{L}_{W}, d_L)\to (sW\oplus \Q,
d_0)$ is a quasi-isomorphism of complexes, where $d_0$
 denotes the
linear part of $d_L$.
\end{lem}

Recall a Lie model for an adjunction space. 
Let $(\mathbb{L}_W, d_L)$ be a minimal Lie model for $X$.
By definition, there exists a quasi-isomorphism
$C^*(\mathbb{L}_W, d_L) \stackrel{\simeq}{\to} A_{PL}(X)$.
Moreover, we have an isomorphism
$\sigma_L \co  H(\mathbb{L}_W, d_L)\stackrel{\cong}{\to}\pi_*(\Omega
X)\otimes \Q$ of graded Lie algebras.  Define an isomorphism
$\tau_L \co  sH(\mathbb{L}_W, d_L)\to \pi_*(X)\otimes \Q$
 by composing the map $\sigma_L$ with the inverse of
the connecting isomorphism $\partial \co  \pi_{*+1}(X)\otimes \Q \to
\pi_{*}(\Omega X)\otimes \Q$. Let $z_{\alpha}$ be a cycle of
$\mathbb{L}_W$ such that $\tau_L$ sends the class $s[z_{\alpha}]
\in sH(\mathbb{L}_W, d_L)$ to
$[\alpha]\in \pi_*(X)\otimes \Q$. 
Then, as a Lie model for
the adjunction space $X\cup_{\alpha}e^{k+1}$,  we can choose the graded
Lie algebra $(\mathbb{L}_{W\oplus \Q\{w_{\alpha}\}}, d)$ with
$d_{|W}=d_L$ and $d(w_{\alpha})=z_{\alpha}$ \cite[Theorem 24.7]{F-H-T}.
By applying the construction described in \fullref{sec:2}, we obtain a Sullivan
model for $\F(X\cup_{\alpha}e^{k+1}, Y)$ of the form
$(\wedge (V\otimes C_*(\mathbb{L}_{W\oplus \Q\{w_{\alpha}\}}, d)), \delta)$.

We need the following lemma to prove \fullref{thm:main}.  

\begin{lem} 
 \label{lem:model-map}
Let 
$$m_1 \co  \Q[V] \to  \Q[V\otimes C_*(\mathbb{L}_{W})]$$
$$m_2\co  \Q[V] \to  \Q[V\otimes C_*(\mathbb{L}_{W\oplus \Q\{w_{\alpha}\}})]
\leqno{\text{and}}$$
be the inclusions of relative Sullivan algebras.  
Let $$\eta \co  \Q[V\otimes C_*(\mathbb{L}_{W})] \to
\Q[V\otimes C_*(\mathbb{L}_{W\oplus \Q\{w_{\alpha}\}})]$$ be
 the map induced by the inclusion 
$(\mathbb{L}_{W}, d) \to (\mathbb{L}_{W\oplus \Q\{w_{\alpha}\}}, d)$ 
of DGL's.  Then there exists a commutative diagram 
\begin{small}  
$$
\xymatrix@C1pt@R15pt{
 &   \Q[V] \ar[rr]^{\simeq}  \ar@{->}'[d]^{m_2}[dd] \ar[dl]_{m_1} 
   & &  A_{PL}(Y) \ar[dd]^{A_{PL}(ev_*)} 
   \ar[dl]_{A_{PL}(ev_*)}  \\
\Q[V\otimes C_*(\mathbb{L}_{W})] \ar[dr]_{\eta} 
   \ar[rr]^(0.6){\simeq} &  & A_{PL}(\F(X, Y)) 
\ar[dr]^{A_{PL}(i^\sharp)} & \\
 &\Q[V\otimes C_*(\mathbb{L}_{W\oplus \Q\{w_{\alpha}\}})] \ar[rr]^{\simeq}  
 & &  A_{PL}(\F(X\cup_{\alpha}e^{k+1}, Y)) 
}
$$
\end{small}%
in the category of DGA's in which three horizontal arrows are
 quasi-isomorph\-isms.  Hence the map  
$\overline{\eta} \co  \Q[V\otimes C_*(\mathbb{L}_{W})^+] \to
\Q[V\otimes C_*(\mathbb{L}_{W\oplus \Q\{w_{\alpha}\}})^+]$ 
 induced by $\eta$  
is a Sullivan model for the map 
$i^\sharp \co  \F_*(X\cup_{\alpha}e^{k+1}, Y) \to \F_*(X, Y)$ 
{\rm \cite[Definition, page 182]{F-H-T}}.
\end{lem}

\begin{proof}
See the appendix. 
\end{proof}

\noindent
\proof[Proof of \fullref{thm:main}]
Under the hypotheses in \fullref{thm:main},
we see that the space $\F_*(X, Y)$ is simply-connected and
$\F_*(X\cup_{\alpha}e^{k+1}, Y)$ is connected.
We shall prove  the fibration \ref{eq:1.1} is rationally trivial if the inequality
$
\text{bl}(\alpha)
>  d_1\text{--depth}(Y)
$
holds.

Under the notation mentioned above,  we assume that
$$
z_\alpha = \sum_i[x_{i_n} [x_{i_{n-1}}[x_{i_{n-2}}, ...,
[x_{i_1}, x_{i_0}], ...]]]
$$
with appropriate cycles $x_{i_j}$ in $\mathbb{L}_W$, where
$n =\text{bl}(\alpha)$. 
By virtue of \fullref{lem:model-map},    
we see that the inclusion
$\overline{\eta} \co  \wedge (V\otimes C_*(\mathbb{L}_{W}, d)^+) \to
\wedge (V\otimes C_*(\mathbb{L}_{W\oplus \Q\{w_{\alpha}\}}, d)^+)$
is a model for the projection $i^\sharp$ of the fibration \ref{eq:1.1}.
Let $\varphi \co (\wedge (Z), d) \to (\wedge (V\otimes C_*(\mathbb{L}_{W}, d)^+),
\delta)$ be the minimal model described before
\fullref{lem:2}. Observe that $\varphi$ is an inclusion and $Z \cong V\otimes H_*(
C_*(\mathbb{L}_{W}, d)^+)  \cong V\otimes sW$. 
If $\wedge (\widetilde{Z}')$
is a  minimal model for the Sullivan algebra  
$(\wedge (V\otimes C_*(\mathbb{L}_{W\oplus \Q\{w_{\alpha}\}}, d)^+),
\delta)$, then $\widetilde{Z}'$ is isomorphic to $V\otimes
H_*(C_*(\mathbb{L}_{W\oplus \Q\{w_{\alpha}\}}, d)^+)$ and hence to $V\otimes
s(W\oplus \Q\{w_{\alpha}\})$.
With this in mind, we define a Sullivan algebra $(\wedge \widetilde{Z},
\widetilde{d})$ by $\widetilde{Z}= V\otimes s(W\oplus \Q \{w_{\alpha}\})
\cong
Z\oplus (V\otimes sw_{\alpha})$, $\widetilde{d}_{|_Z} = d$ and
$\widetilde{d}_{|_{V\otimes sw_{\alpha}}}\equiv 0$. In order to prove
Theorem 1.2, it suffices to show that there exists a quasi-isomorphism
$
\psi  \co  (\wedge \widetilde{Z},  \widetilde{d}) \to (\wedge (V\otimes
C_*(\mathbb{L}_{W}, d)^+), \delta)
$
such that the diagram
$$
\xymatrix@=10pt{
   (\wedge Z, d)\ar@{^{(}->}[dd]_{\simeq}^{\varphi}
                      \ar[rr]^{I}
                               & & (\wedge {\widetilde{Z}}, \widetilde{d}) 
               \ar[dd]^{\simeq}_{\psi} \\
                     \\
   (\wedge (V\otimes C_*(\mathbb{L}_{W}, d)^+), \bar{\delta})
      \ar[rr]_(.45){\overline{\eta}} & &
(\wedge (V\otimes C_*(\mathbb{L}_{W\oplus \Q\{w_{\alpha}\}}, d)^+),
   \bar{\delta})
 }
$$
is commutative,  where $I$ is the inclusion. In fact, we then see that 
the map $I$ is regarded as a Sullivan model for $i^\sharp$. 
Moreover the Sullivan algebra
$(\wedge \widetilde{Z},  \widetilde{d} )$ 
is isomorphic to $(\wedge Z, d)\otimes 
(\wedge(V\otimes sw_\alpha), 0)$ as a DGA.  
Observe that $(\wedge(V\otimes sw_\alpha), 0)$ is the minimal model
for $\Omega^{k+1}Y$.

We shall construct the required map $\psi$. 
Put $\wedge U = \wedge (V\otimes C_*(\mathbb{L}_{W\oplus
\Q\{w_{\alpha}\}},d)^+)$. Let $\wedge^sU$ be the vector subspace of
$\wedge U$ consisting of elements with wordlength $s$ and 
$\wedge^{\geq s}U$ the ideal of $\wedge U$ generated by $\wedge^sU$.   
Assume that $v \in V_m$, where $m = d_1\text{--depth}(Y)$.
We first choose a cycle
$$
c_\alpha= sw_\alpha-  \sum_i sx_{i_n}\wedge
s[x_{i_{n-1}}[x_{i_{n-2}}, ..., [x_{i_1}, x_{i_0}] ...]]
$$
in $C_*(\mathbb{L}_{W\oplus \Q\{w_{\alpha}\}},d)$
and define an element $\gamma_1$ of $\wedge U$ 
by $\gamma_1=v\otimes c_\alpha$.
Observe that $n > m$ by assumption. We set
$x_{{i_{n-1}}, ..,{i_0}} =
[x_{i_{n-1}}[x_{i_{n-2}}, ...,[x_{i_1}, x_{i_0}] ...]]$.
It follows from \ref{eq:2.1} that, in $\wedge^{\geq 2}U$,
\begin{multline*}
\overline{\delta}(\gamma_1) =  - \ \big(
\sum_{i,j_1}(-1)^{|sx_{i_n}||v_{j_1}'|}(v_{j_1}\otimes sx_{i_n})\cdot
(v_{j_1}'\otimes sx_{{i_{n-1}}, ..,{i_0}})  \\
   + \sum_{i,j_1}(-1)^{|sx_{i_{n-1}, ..,i_0}|
|sx_{i_n}|+|sx_{{i_{n-1}}, ..,{i_0}}||v_{j_1}'|}
(v_{j_1}\otimes sx_{i_{n-1}, ..,i_0}) \cdot
(v_{j_1}'\otimes sx_{i_n} ) \big)
\end{multline*}
if $d_1(v)= \sum_{j_1}v_{j_1}v_{j_1}'$.
We see that $\overline{\delta}(\gamma_1)$ belongs to
$\wedge^2U$ and is determined without depending on the term of
$(d- d_1)(v)$ because $sx_{i_n}$ and $sx_{{i_{n-1}}, ..,{i_0}}$
are primitive. Observe that $v_{j_1}$ and $v_{j_1}'$ are in $V_{m-1}$
(see \fullref{lem:key1} for more polished result on the image of
$d_1$).

We next define an element $\gamma_2 \in \wedge^{2}U$ by
\begin{eqnarray*}
\gamma_2 &=&
\sum_{i,j_1}(-1)^{\varepsilon_{i_n, ..,i_0}}
(v_{j_1}\otimes sx_{i_n})\cdot
(v_{j_1}'\otimes sx_{i_{n-1}}\wedge sx_{i_{n-2}, ..,{i_0}}) \\
& & \qquad + \sum_{i,j_1}
(-1)^{\varepsilon_{i_n,..,i_0}'}
(v_{j_1}\otimes sx_{i_{n-1}}\wedge sx_{i_{n-2}, ..,i_0}) \cdot
(v_{j_1}'\otimes sx_{i_n}),
\end{eqnarray*}
where  $\varepsilon_{i_n, ..,i_0}$ and  $\varepsilon_{i_n, ..,i_0}'$
denote the integers
$|sx_{i_n}||v_{j_1}'|+|v_{j_1}\otimes sx_{i_n}|+|v_{j_1}'|+|sx_{i_{n-1}}|$
and $|sx_{i_{n-1}, ..,i_0}||sx_{i_n}|
+|sx_{{i_{n-1}}, ..,{i_0}}||v_{j_1}'|+|v_{j_1}|+|sx_{i_{n-1}}|$,
respectively.
Since $ sx_{i_n}$ is primitive, it follows from \fullref{lem:1} that
 $\overline{\delta}(\gamma_1) = - \overline{\delta}(\gamma_2)$ in
$\wedge^2U$.

In a similar fashion, we can define elements $\gamma_l\in \wedge^{l}U$ 
so that  $\overline{\delta}(\gamma_{l-1})= -
\overline{\delta}(\gamma_{l})$
in  $\wedge^{l}U$ and each term of $\gamma_l$ has the form
$$
y\cdot \big(v_{j_l} \otimes (sx_{i_{n-l+1}}\wedge sx_{i_{n-l}, ..,i_0})\big),
$$
where $v_{j_l}\in V_{m-l+1}$ and $y$ is an element in the ideal of $\wedge
U$
generated by elements of the form $u\otimes sx_{i_s}$ for some $u\in V$.
Since $\overline{\delta}(\gamma_l)\in \wedge^{l}U\oplus
\wedge^{l+1}U$ and $\overline{\delta}(\gamma_{m+1}) =0$ in $\wedge^{m+2}U$,
it follows that $\gamma_v:=\gamma_1 +\cdots +\gamma_{m+1}$ is
a $\overline{\delta}$--cycle in $\wedge U$ (see \ref{eq:3.4} below in
which $\overline{\delta}_1$ denotes the linear part of the differential
$\overline{\delta}$ and
$\overline{\delta}_2= \overline{\delta}-\overline{\delta}_1$).
\begin{equation}\label{eq:3.4}
\xymatrix@R-12pt{
0 &  \\
\gamma_{1} \ar@{->}[u]^{\overline{\delta}_1}
 \ar@{->}[r]^{\overline{\delta}_2}&  0 \\
                                 &    \gamma_{2}
                          \ar@{->}[u]^{\overline{\delta}_1}
                    \ar@{.>}[r]^{\overline{\delta}_2} & 0 \\
                                 &                   &  \gamma_{m+1}
          \ar@{->}[r]^{\overline{\delta}_2}
                    \ar@{.>}[u]^{\overline{\delta}_1} & 0
}
\end{equation}
Observe that the element $\gamma_2 +\cdots +\gamma_{m+1}$ can be regarded as the
element $x_{ij}$ in condition \ref{eq:3.1}.

The same argument above works well to show that $v\otimes sw_{\alpha}$
is a cycle when $v\in V_l$ for $l < m$
since $\text{bl}(\alpha)=n > m =d_1\text{--depth}(Y)$.

We here define a map 
$\psi \co(\wedge \widetilde{Z},  \widetilde{d}) \to (\wedge
(V\otimes C_*(\mathbb{L}_{W\oplus {\mathbb Q}\{w_{\alpha}\}}, d)^+),
\delta)$ 
by $\psi_{|_Z}= \overline{\eta} \varphi$ and
$\psi(v\otimes sw_{\alpha}) = \gamma_v$ for $v\otimes sw_\alpha \in
V\otimes sw_\alpha$. The construction of $\Q[w_{ij}]$
described before \fullref{lem:2} tells us that $\psi$ is a minimal
model.
Moreover we see that all the required conditions for $\psi$ hold.
This completes the proof of \fullref{thm:main}.
\endproof

\begin{exa}
Let us consider the projective space
$\mathbb{C}P^2 = S^2 \cup_{\gamma}e^4$, where
$\gamma$ denotes the Hopf map. 
Let $Y$ be a $4$--connected space with a  minimal model
$(\wedge V, d)$ for which $V$ is a vector space with a basis
$\{x_1, x_2, x_3, y\}$, $d(x_i)=0$ and  $d(y)=x_1x_2x_3$. Since $\gamma$
 is decomposable in $\pi_*(S^2)\otimes \Q$, 
it is evident that  $\text{bl}(\gamma)=\text{bl}([\iota, \iota])= 1 
> 0= d_1\text{--depth}(Y)$, where $\iota$ is the generator of
 $\pi_2(S^2)$. 
Thus \fullref{thm:main} allows us to conclude that
the fibration $\Omega^4Y \to \F_*(\mathbb{C}P^2, Y) \to \Omega^2Y$ 
is rationally trivial. 
\end{exa}

\begin{exa} 
Let ${\mathcal L}P^2$ be the Cayley plane and $\mathbb{C}P^2_i$ a
 copy of the complex projective plane for $i = 1, 2$. Let 
$\iota_i$ denote the generator of $\pi_2(\mathbb{C}P^2_i)$. 
The space 
$\mathbb{C}P^2_1\vee \mathbb{C}P^2_2\cup_{[\iota_1, \iota_2]}e^4$ 
has a CW--decomposition for which the bracket length of each attaching map  
is greater than or equal to $1$.  
Since $H^*({\mathcal L}P^2; \Q)\cong \Q[x_8]/(x_8^3)$, where $\deg x_8 = 8$, 
it follows that 
$\text{WL}({\mathcal L}P^2)= d_1\text{--depth}({\mathcal L}P^2)=0$. 
\fullref{cor:1} yields that, for any based space $Z$,  
\begin{eqnarray*}
[Z\wedge (\mathbb{C}P^2_1\vee \mathbb{C}P^2_2\cup_{[\iota_1, \iota_2]}e^4), 
{\mathcal L}P^2_\Q]_* &\cong& (H^{8-4}(Z; \Q)\oplus H^{23-4}(Z; \Q))^{\oplus 3}\\
& & \ \ \ \ \oplus (H^{8-2}(Z; \Q)\oplus H^{23-2}(Z; \Q))^{\oplus 2}. 
\end{eqnarray*}
\end{exa}

\begin{exa} Let $G$ and $H$ be a compact connected Lie group and a closed
 subgroup of $G$, respectively. 
By considering the K.S--extension of the fibration $G\to G/H\to BH$, we see
that the minimal model  $(\wedge V, d)$ for $G/H$ satisfies the condtions:
$dV^{even}=0$ and $dV^{odd}\subset \wedge V^{even}$.
This implies that $d_1\text{--depth}(G/H) \leq 1$. 
Let $X$ and $\alpha \co  S^k \to X$ be as in \fullref{thm:main}.   
Suppose that $\text{Conn}(G/H)\geq  \text{max}\{k+1,  \dim X\}$. 
Then the fibration 
$$
\Omega^{k+1}Y=\F_*(S^{k+1}, G/H) \stackrel{j^\sharp}{\to}
\F_*(X\cup_{\alpha}e^{k+1}, G/H)  \stackrel{i^\sharp}{\to}
\F_*(X, G/H) 
$$
is rationally trivial if $\text{bl}(\alpha)>1$. 
\end{exa}

\begin{exa}
Recall from \cite{F-H-T} that a simply-connected space $Y$ is elliptic
if $\dim \pi_*(Y)$ $\otimes \Q < \infty$ and $\dim H_*(Y; \Q) < \infty$.
Let $Y$ be an $n$--connected finite dimensional elliptic CW complex
with a minimal model $(\wedge V, d)$.  Let $\{v_i\}$ be a basis of
$V$.  If $v_{i_s}\in V_s-V_{s-1}$, then $\deg v_{i_s}\geq (s+1)n+1$
(see the \fullref{sec:2} for the notation $V_s$). Put
$m=d_1\text{--depth}(Y)$ and let $v$ be an element of $V$ with the
maximal degree.  Then $\deg v$ is odd from Friedlander--Halperin
\cite[Theorem 1 and Lemma 2.5]{F-H}.  Therefore it follows from
\cite[Corollary 1.3(3)]{F-H} that
$$
(m+1)n+1\leq \deg v_{i_m}\leq \deg v\leq \sum_{j:odd}j\cdot \dim V^j\leq
2\dim Y-1
$$
and hence
 $2\dim Y/n>m+1=d_1\text{--depth}(Y)+1$. \fullref{thm:main} enable us to
conclude that the fibration \ref{eq:1.1} is rationally trivial 
if $\text{bl}(\alpha)+1 \geq 2\dim Y/\text{Conn}(Y)$. 
\end{exa}

We give examples which assert that the decomposition
in \fullref{thm:main} does not hold in general when
$\text{bl}(\alpha)\leq \text{WL}(Y)$.
To this end, we here recall the result \cite[Theorem 1.2]{Ko} due to
Kotani. 

Let $(\wedge V, d)$ be a minimal model 
for a simply-connected space $Y$. 
Consider the decomposition  $d= d_1+d_2+ \cdots$ of
the differential $d$ as in \fullref{sec:2}.
The $d$--{\it length} of $Y$, denoted $d$--length(Y), is the least integer $m$
such that $d_i\equiv 0$ for $i< m-1$ and $d_{m-1}\not\equiv 0$. Observe that
the $d$--\text{length} of $Y$ is a topological invariant
(see \cite[Theorem 1.1]{Ko}).
As usual, we define the cup-length
of a space $X$, $c(X)$, by the greatest integer $n$ such that
there are elements $\alpha_1$, ..., $\alpha_n$ in $H^+(X; \Q)$ for which
$\alpha_1\cup \cdots \cup \alpha_n \neq 0$.
Then the main result in \cite{Ko} is stated as follows.

\begin{thm}{\rm\cite[Theorem 1.2]{Ko}}\label{thm:Ko}\qua
Let $X$ be a path connected, finite dimensional CW complex
 and $Y$ a connected space  with
$\text{\em Conn}(Y)\geq \dim X$. Suppose that $X$ is formal.
Then the cohomology algebra $\hq{\F_*(X, Y)}$ is a free algebra
if and only if  $d\text{\em -length}(Y) > c(X)$.
\end{thm}

\begin{exa}
 Consider the projective space
$\mathbb{C}P^3=\mathbb{C}P^2 \cup_\alpha e^6$. We observe
 that $\alpha$ is indecomposable in $\pi_*(\mathbb{C}P^2)\otimes \Q$.
Since $d\text{-length}(Y)= 3 =c(\mathbb{C}P^3)$,
it follows from \fullref{thm:Ko} that $\hq{\F_*(\mathbb{C}P^3, Y)}$
is not free. Thus  $\F_*(\mathbb{C}P^3, Y)$ is not rationally
 homotopy equivalent to the product
$\F_*(\mathbb{C}P^2, Y)\times \Omega^6 Y$ because
$\hq{\F_*(\mathbb{C}P^2, Y) \times\Omega^6 Y}$ is free.
Observe that $\text{bl}(\alpha)= 0 = d_1\text{--depth}(Y)$ in this case.
\end{exa}

\begin{exa}
Let  $(\wedge V, d)=(\wedge (x,y),d)$ be the minimal model for  $S^6$, where
 $\deg x=6$, $\deg y=11$, $dx=0$ and $dy=x^2$.
Consider  the fibration
$
\Omega^4 S^6 \stackrel{j^\sharp}{\to} \F_*(\mathbb{C}P^2,S^6)
\stackrel{i^\sharp}{\to} \Omega^2 S^6
$
which is induced from the cofibre sequence
$S^2 \stackrel{i}{\to} \mathbb{C}P^2= S^2 \cup_{\gamma}e^4
  \stackrel{j}{\to} S^{4}$. Let $\iota$ be the generator in
 $\pi_2(S^2)\otimes \Q$. 
Observe that $\gamma= q[\iota, \iota]$ for some nonzero rational number $q$. 
We can choose  
$\Q[V\otimes C_*(\mathbb{L}_{\Q\{\tilde{\iota}, w_{\gamma}\}}, d)^+]$
as a Sullivan model for the function space 
$\F_*(\mathbb{C}P^2, S^6)$, where $\tilde{\iota}$ denotes the element
in $\pi_1(\Omega S^2)\otimes \Q$ corresponding to $\iota$  via the connecting
isomorphism
$\pi_2(S^2)\otimes \Q \to \pi_{1}(\Omega S^2)\otimes \Q$.
 Put $v_4=x\otimes s\tilde{\iota}$,
  $v_{9}=y\otimes s\tilde{\iota}$,
 $v_2=x\otimes (sw_{\gamma}-q(s\tilde{\iota}\wedge s\tilde{\iota}))$
 and $v_7=y\otimes (sw_{\gamma}- q(s\tilde{\iota}\wedge
s\tilde{\iota}))$.
  Then a model for the above fibration is given by
 $$(\wedge (v_4,v_{9}),0)\to
 (\wedge (v_4,v_{9},v_2,v_7),\overline \delta)
 \to (\wedge (v_2,v_7),0)$$
 where $\overline \delta (v_7)=-2q{v_4}^2$
 and $\overline \delta (v_i)=0$
 for $i\neq 7$ (see the proof of \fullref{thm:main} for the
construction).
 Therefore the fibration is not rationally trivial.
It is readily seen that
$\text{bl}([\iota, \iota])= 1= d_1\text{--depth}(S^6)$ in this case.
 \end{exa}

\begin{exa}
Let $Y$ be a $6$--connected space whose minimal model has the trivial
differential. Then the differentials of the minimal models for the spaces
 $\F_*(\mathbb{C}P^2, Y)$ and   $\F_*(\mathbb{C}P^3, Y)$ are also trivial.
Moreover we see that
$
\Omega^6 Y \stackrel{j^\sharp}{\to} \F_*(\mathbb{C}P^3, Y)
\stackrel{i^\sharp}{\to}
\F_*(\mathbb{C}P^2, Y)
$
is rationally trivial though  $\text{bl}(\alpha)= 0= d_1\text{--depth}(Y)$.
This fact implies that the converse assertion of \fullref{thm:main}
 does not hold in general.
\end{exa}

\medskip
{\bf Acknowledgements}\qua
The authors are grateful to Jean-Claude Thomas for useful comments
on the Whitehead length and the $d_1$--depth.
They would like to thank
Yasusuke Kotani for explaining his result \cite[Theorem 1.2]{Ko}
and wish to express their thanks to the referee of the previous version
of this article. His comments lead us to the examples described in
\fullref{sec:3}.

\section{Appendix}\label{sec:4}

We prepare to prove \fullref{thm:d_1}.
Let $(\wedge V, d)$ be the minimal model for a simply-connected space $Y$.
Recall the graded Lie algebra $L$
associated with  a minimal model  $(\wedge V, d)$ for $Y$
(see \cite[Section 21, (e)]{F-H-T}).
The graded vector space $L$ is defined by
$sL=\text{Hom}(V, \Q)$.
We define a pairing $\langle \ ;\ \rangle \co  V\times sL \to \Q$ by
$\langle v;sx \rangle =(-1)^{\deg v}sx(v)$.
Moreover, using the pairing, define a trilinear map
$$
\langle \ ;\ , \ \rangle \co  \wedge^{2}V \times sL \times sL \to \Q
$$
by $\langle v\wedge w ;sx , sy \rangle =
\langle v;sx\rangle \langle w;sy \rangle
+(-1)^{|v||w|} \langle w;sx\rangle \langle v;sy\rangle$.
Then the Lie bracket $[ \ , \ ]$ in $L$ is given by requiring that (4.1):
$$
\langle v ; s[x,y]\rangle  = (-1)^{\deg y+1}\langle d_1 v ; sx, sy
\rangle
$$
for $x, y \in L$ and $v\in V$.
The result \cite[Theorem 21.6]{F-H-T} asserts that $L$ is isomorphic to
the homotopy Lie algebra $L_Y$. Therefore,
in order to prove \fullref{thm:d_1},
it suffices to show that  the $d_1$--depth of $Y$ is
equal to the integer $\text{WL}(L)$, which is the greatest integer $n$
such that $[L, L]^{(n)}\neq 0$. 
As in the proof of \fullref{thm:main}, we may write
$x_{i_n,..,i_0}$ for the element
$[x_{i_n} [x_{i_{n-1}}, ...,[x_{i_1}, x_{i_0}]]]$ in $L$.

\begin{lem}
 \label{lem:5.2}
 For any $\alpha \in V_{n-1}$ and any $x_{i_n,..,i_0}\in [L,L]^{(n)}$,
$\langle \alpha, sx_{i_n, ...,i_0}\rangle = 0$.
\end{lem}

\begin{proof}
We argue by induction on $n$.
From the formula (4.1), we see that
$\langle \alpha, sx_{i_1,i_0}\rangle = 0$ for any
$\alpha \in V_{0}$.
Suppose that $\langle \beta, sx_{i_{n-1}, ...,i_0}\rangle = 0$ for any
$\beta \in V_{n-2}$.  Let $\alpha$ be an element of $V_{n-1}$.
Then we can write
$d_1(\alpha) =\sum_j\beta_j\beta_j'$ with some elements
$\beta_j$ and $\beta_j'$ of $V_{n-2}$.
Thus it follows from the definition of the trilinear map
$\langle \ ;\ , \ \rangle$  that
$$\eqalignbot{
 \langle \alpha, sx_{i_n, ...,i_0}\rangle
  &=  \pm  \langle d_1\alpha;
sx_{i_n}, s[x_{i_{n-1}},... [x_{i_1},x_{i_0}]] \rangle   \cr
  &= \pm \langle \sum_j\beta_j\beta_j';
        sx_{i_n}, s[x_{i_{n-1}},... [x_{i_1},x_{i_0}]]\rangle  \ = \ 0.} 
\proved$$
\end{proof}

\begin{prop}
\label{prop:white1}
$d_1\text{\em --depth}(Y) \geq \text{\em WL}(L)$.
\end{prop}

\begin{proof}
Suppose that $[L,L]^{(m)}\neq 0$. We choose a nonzero element
$x_{i_{m}, ...,i_0}$ of $[L,L]^{(m)}$.
Let $v_m$ be an element of
$V$ such that $\langle v_m, sx_{i_m, ...,i_0}\rangle \neq 0$.
\fullref{lem:5.2} yields that $v_m \not\in V_{m-1}$ and hence
the $d_1\text{--depth}(Y) \geq m$. 
\end{proof}

In order to complete the proof of \fullref{thm:d_1}, it remains to
prove that $d_1\text{--depth}(Y)$ is less than or equal to
$\text{WL}(L)$.
To this end, we first characterize the vector space $V_0$ using the
space $S$ of indecomposable elements of $L$. One can express the vector
space
as $L=S\oplus [L, L]$.

\begin{lem}
 \label{lem:sS}
 $sS = \text{\em Hom} (V_0, \Q)$.
\end{lem}

\begin{proof} Let $\{x_i\}$ and $\{y_j\}$ be bases for $S$ and
$[L, L]$, respectively.
Let $\{(sy_j)^*\}\cup  \{(sx_i)^*\}$ be the basis of $V$ which
is the dual to the basis $\{sy_j\}\cup  \{sx_i\}$ of $sL$.
It suffices to prove that $V_0$ is the vector space spanned by
$\{(sx_i)^*\}$.
Since  $\langle d(sx_i)^*; sx, sy\rangle
= \langle (sx_i)^*; s[x, y]\rangle =0 $ for any $x,y \in V$, it follows that
$(sx_i)^*\ \in V_0$.
For any $v \in V_0$, we write $v= \sum_i\lambda_i(sx_i)^*
+\sum_j\mu_j(sy_j)^*$ and $sy_j= \sum_{k_j}s[a_{k_j}, b_{k_j}]$ for some
$a_{k_j}$ and $b_{k_j}$ in $L$.
It follows  that
\begin{eqnarray*}
0 \ = \ \sum_{k_j}\langle dv; sa_{k_j}, sb_{k_j}\rangle &=&
 \langle  \sum_i\lambda_i(sx_i)^*
+\sum_j\mu_j(sy_j)^* , \sum_{k_j}s[a_{k_j}, b_{k_j}]\rangle \\
&=&\langle  \sum_i\lambda_i(sx_i)^*
+\sum_j\mu_j(sy_j)^* , sy_j \rangle \ = \ \mu_j . 
\end{eqnarray*}
Thus we have  $v= \sum_i\lambda_i(sx_i)^*$.
\end{proof}

We here study a fundamental property of
the quadratic part of the differential $d$.
Write $V_n=\overline{V_n}\oplus V_{n-1}$ and
fix a basis $\{w_j\}$ for $\overline{V_n}$.

\begin{lem}
 \label{lem:key1}
For any $u \in V_{n+1}$,
there exist elements $e_j \in V_0$ and $f_s, g_s\in V_{n-1}$ such
 that
$$
d_1u= \sum_je_jw_j +\sum_sf_sg_s.
$$
\end{lem}

\begin{proof}
The result for $n =0$ is immediate. Let us assume that $n\geq 1$.
We can write
$$
d_1u =\sum_{i\leq j}\lambda_{ij}w_iw_j + \sum_je_jw_j + \sum_sf_sg_s
$$
with some elements $e_j,  f_s$,  $g_s\in V_{n-1}$ and $\lambda_{ij}\in \Q$.
By applying the differential $d_1$ to the equality, we have
\begin{eqnarray*}
0 = d_1d_1u &=&\sum_{i\leq j}\lambda_{ij}d_1(w_i)w_j +
                   \sum_{i\leq j}(-1)^{|w_i|}\lambda_{ij}w_id_1(w_j)
   + \sum_jd_1(e_j)w_j + Z \\
           &=& \sum_j\big(\sum_i \mu_{ij}d_1w_i + d_1e_j \big)w_j + Z
\end{eqnarray*}
in which $\mu_{ii} = 2 \lambda_{ii}$,  $\mu_{ij} = \lambda_{ij}$ for $i<j$,
$\mu_{ij} = (-1)^{|w_j|+|w_j||d_1w_i|}\lambda_{ij}$ for $i>j$ and
$Z$ is an appropriate element of $\wedge^{\geq 2}V_{n-1}$. Thus
we see that $\sum_i \mu_{ij}d_1w_i + d_1e_j=0$ for any $j$.
Since $d_1e_j \in \wedge V_{n-2}$,
it follows that $\sum_i \mu_{ij}w_i$ is in $V_{n-1}$ and hence
$\sum_i \mu_{ij}w_i=0$. The fact enables us to conclude that
$\mu_{ij}= 0$ for any $i$, $j$ and that $e_j $ is in $V_0$.
We have the result.
\end{proof}

\fullref{lem:sS} allows us to
choose a basis $\{sx_k\}_{k\in J}$ for $sS$ and its dual basis
$\{e_k\}_{k\in J}$ for $V_0$. Let $\{w_m\}_{m\in M}$ be a basis for
$\overline{V_1}$.
We can write $d_1w_m
 = \sum_{k_1,k_0}\lambda_{k_1,k_0}^{(m)}e_{k_1}e_{k_0}$, where
$\lambda_{k_0,k_0}^{(m)} =0$ if $|e_{k_0}|$ is odd.

\begin{lem}
 \label{lem:key2}
 Let $\{v_p^{(n)}\}_{1\leq p\leq l_n}$ be a basis for $\overline{V_n}$,
 where $n\geq 1$.
Then there exist rational numbers $\theta_{k_n,..,k_2,m}^{v_p^{(n)}}$
for all $k_n,..,k_2$ and $m$ such that
\begin{small}  
$$
(-1)^{|s[x_{k_{n-1}}[ ...,[x_{k_1}, x_{k_0}]...]]|}
    \langle
      v_p^{(n)}, s[x_{k_n},[x_{k_{n-1}}[ ...,[x_{k_1}, x_{k_0}]...]]]
         \rangle
= \sum_m \theta_{k_n,..,k_2,m}^{v_p^{(n)}}\lambda_{k_1,k_0}^{(m)}
$$
\end{small}%
and the matrix $\big(  \theta_{k_n,..,k_2,m}^{v_p^{(n)}}   \big)$ with
$l_n$ columns is of full rank; that is, the column vectors obtained from
 the matrix are linearly independent.
Here, we regard the set $\{(k_n,..,k_2,m)\}$ as the ordered set $\{I_i\}$
 by using the lexicographic order on elements $(k_n,..,k_2,m)$. Then the
$(i,p)$ component of
the matrix $\big(  \theta_{k_n,..,k_2,m}^{v_p^{(n)}}   \big)$
is given by $\theta_{I_i}^{v_p^{(n)}}$.
\end{lem}

\begin{proof}
We argue by induction on $n$.
In the case where $n= 1$, the result is immediate. 

We assume that $n \geq 2$ and that the assertion is true up to $n$.
To simplify, we write $v_p$ for $v_p^{(n+1)}$.
Thanks to 
 \fullref{lem:key1}, we can express
$$
d_1v_p = \sum_{1\leq k \leq q, 1\leq j\leq r}\mu_{kj}^{v_p}e_kv_j^{(n)}
           + \sum_sf_sg_s
$$
with some elements $f_s$ and $g_s$ in $V_{n-1}$, where $\mu_{kj} ^{v_p} \in
\Q$.
Then it follows that
\begin{eqnarray*}
(-1)^\varepsilon \langle v_p, sx_{k_{n+1}, ...,k_0} \rangle
= \langle \sum_{k, j}\mu_{kj}^{v_p}e_kv_j^{(n)} + \sum_sf_sg_s
            \ ; \ sx_{k_{n+1}},  sx_{k_{n}, ...,k_0} \rangle =:\theta,
\end{eqnarray*}
where $\varepsilon =|sx_{k_{n}, ...,k_0}|$.
\fullref{lem:5.2}  allows us to deduce that
\begin{eqnarray*}
\theta &=& \sum_{kj} \mu_{kj}^{v_p} \langle e_k, sx_{k_{n+1}} \rangle
          \langle v_j^{(n)} ,  sx_{k_{n}, ...,k_0}     \rangle \\
       &=& \sum_j  \mu_{k_{n+1}j}^{v_p}\big(\sum_m
               \theta_{k_n,..,k_2,m}^{v_j^{(n)}}\lambda_{k_1,k_2}^{(m)}\big)
 = \sum_m
   \big(\sum_j  \mu_{k_{n+1}j}^{v_p} \theta_{k_n,..,k_2,m}^{v_j^{(n)}}\big)
                                    \lambda_{k_1,k_2}^{(m)}
\end{eqnarray*}
We put $\theta_{k_{n+1},..,k_2,m}^{v_p} =
 \sum_j  \mu_{k_{n+1}j}^{v_p}\theta_{k_n,..,k_2,m}^{v_j^{(n)}}$ and
consider the matrix 
$\big(  \theta_{k_{n+1},..,k_2,m}^{v_p}   \big)$.
Then, by definition, we see that the matrix is decomposed as
$$
\Big(  \theta_{k_{n+1},..,k_2,m}^{v_p}   \Big) =
\Big( \theta_{k_{n+1},I_i}^{v_p}  \Big) =
\left(
 \begin{array}{ccc}
  & \theta_{1I_1}^{v_p} &   \\
  &       \vdots  &  \\
 & \theta_{1I_s}^{v_p} &  \\
\ldots  & \theta_{2I_1}^{v_p} & \ldots \\
 &       \vdots  & \\
  &       \vdots  &\\
 & \theta_{qI_s}^{v_p} &
\end{array}
\right)
=
\left(
 \begin{array}{cccc}
A &   &   &  \\
  & A &   &   \\
  &   & \ddots  &   \\
  &   &   & A
\end{array}
\right) B,
$$
where
$$
A = \Big( \theta_{I_i}^{v_j^{(n)}} \Big) \quad \text{and} \quad
B =
\left(
 \begin{array}{cccc}
  & \mu_{11}^{v_p} &   \\
  &       \vdots  &  \\
 & \mu_{1r}^{v_p} &  \\
\ldots  & \mu_{21}^{v_p} & \ldots \\
 &       \vdots  & \\
 & \mu_{qr}^{v_p} &
\end{array}
\right) .
$$
Since the set $\{v_p\}$ is a basis for $\overline{V_{n+1}}$, it follows
 that the matrix $B$ is of full rank. By assumption, $A$ is of full rank
 and hence so is $\big(  \theta_{k_{n+1},..,k_2,m}^{v_p}   \big)$.
This completes the proof.
\end{proof}

\fullref{thm:d_1} follows from \fullref{prop:white1} and
the following proposition.

\begin{prop}
 \label{prop:white2}
$d_1\text{\em --depth}(Y) \leq \text{\em WL}(L)$.
\end{prop}

\begin{proof}
Put $n =d_1\text{--depth}(Y)$. It suffices to prove that the
 inequality holds  in the case where  $n\geq 1$.
Let $\{v_p^{(n)}\}_{1\leq p\leq l_n}$ be a basis for $\overline{V_n}$.
We assume that
$$\langle v_p^{(n)},
    s[x_{k_n},[x_{k_{n-1}}[ ...,[x_{k_1}, x_{k_0}]...]]]
         \rangle = 0
$$
for any $k_n, ..., k_1, k_0$.  Then it is readily seen that
$\sum_m \theta_{k_n,..,k_2,m}^{v_p^{(n)}}\lambda_{k_1,k_0}^{(m)}=0$,
where  $\theta_{k_n,..,k_2,m}^{v_p^{(n)}}$ are rational numbers described in
\fullref{lem:key2}.
Consider the linear combination
$\sum_m \theta_{k_n,..,k_2,m}^{v_p^{(n)}}w_m$ with the basis $\{w_m\}$ for
 $\overline{V_1}$. We have
\begin{eqnarray*}
d_1(\sum_m \theta_{k_n,..,k_2,m}^{v_p^{(n)}}w_m)& =&
  \sum_m \theta_{k_n,..,k_2,m}^{v_p^{(n)}}d_1(w_m) \\
  &=& \sum_m \theta_{k_n,..,k_2,m}^{v_p^{(n)}}
\sum_{k_1,k_0}\lambda_{k_1,k_0}^{(m)}e_{k_1}e_{k_0} \\
&=& \sum_{k_1,k_0}(\sum_m \theta_{k_n,..,k_2,m}^{v_p^{(n)}}
\lambda_{k_1,k_0}^{(m)})e_{k_1}e_{k_0} = 0.
\end{eqnarray*}
It follows that $\sum_m \theta_{k_n,..,k_2,m}^{v_p^{(n)}}w_m \in V_0$
and hence $\theta_{k_n,..,k_2,m}^{v_p^{(n)}}= 0$ for any $m$. Consequently,
$\theta_{k_n,..,k_2,m}^{v_p^{(n)}}$ is zero for any $m, k_n, ..., k_2$,
which is a contradiction.
\end{proof}

In the rest of this section, we shall prove \fullref{lem:model-map}. 
To this end, we first prepare a lemma.  
\begin{lem}
 \label{lem:sullivan-model}
 The map $\eta \co  \Q[V\otimes C_*(\mathbb{L}_{W})] \to
\Q[V\otimes C_*(\mathbb{L}_{W\oplus \Q\{w_{\alpha}\}})]$ in 
\fullref{lem:model-map} is the inclusion of a relative Sullivan
 algebra.   
\end{lem}

\begin{proof}
We write $\mathbb{L}_{W\oplus \Q\{w_{\alpha}\}}= \mathbb{L}_{W}\oplus Z$
 with appropriate vector space $Z$. 
Then the $C_*(\mathbb{L}_{W\oplus \Q\{w_{\alpha}\}})$ is decomposed as 
$C_*(\mathbb{L}_{W\oplus \Q\{w_{\alpha}\}})= \wedge(s\mathbb{L}_{W})\otimes 
\wedge(sZ)= \wedge(s\mathbb{L}_{W})\otimes 1 \oplus 
\wedge(s\mathbb{L}_{W})\otimes \wedge(sZ)^+$. We see that 
$V\otimes C_*(\mathbb{L}_{W\oplus \Q\{w_{\alpha}\}})= 
 V\otimes C_*(\mathbb{L}_{W})\oplus V\otimes U$ and hence 
$\Q[V\otimes C_*(\mathbb{L}_{W\oplus \Q\{w_{\alpha}\}})]=
\Q[V\otimes C_*(\mathbb{L}_{W})]\otimes \Q[V\otimes U]$, where 
$U =\wedge(s\mathbb{L}_{W})\otimes \wedge(sZ)^+$.  Let $U_{(j)}$ be 
the vector subspace of $U$ consisting of elements with ordinary 
homology degree $j$, namely  
$U_{(j)}=(\wedge(s\mathbb{L}_{W})\otimes \wedge(sZ)^+)_j$. 
Put $V(k)= \oplus_{i+j\leq k}V_{ij}$, 
where $V_{ij}=V^i\otimes U_{(j)}$. 
It is readily seen that $\cup_kV(k)= V\otimes U$ and 
$\delta(V(k))\subset \Q[V\otimes C_*(\mathbb{L}_{W})]\otimes 
\Q[V(k-1)]$. Thus we have the result.  
\end{proof}

\noindent
\proof[Proof of \fullref{lem:model-map}]
Let $i \co  X \to X\cup_{\alpha}e^{k+1}$ be the inclusion map and 
$l \co  C_*(\mathbb{L}_{W}) 
\to C_*(\mathbb{L}_{W\oplus \Q\{w_{\alpha}\}})$ 
the DGC map induced by the natural inclusion 
$\mathbb{L} \to \mathbb{L}_{W\oplus \Q\{w_{\alpha}\}}$. 
Then there exists a homotopy commutative diagram 
$$
\xymatrix@C25pt@R15pt{
A_{PL}(X) & A_{PL}(X\cup_{\alpha}e^{k+1}) \ar[l]_(0.6){A_{PL}(i)} \\
C^*(\mathbb{L}_{W}) \ar[u]^{\simeq}& 
C^*(\mathbb{L}_{W\oplus \Q\{w_{\alpha}\}}) \ar[l]_{l^*}\ar[u]_{\simeq}, 
}
$$
where two vertiacal arrows 
are quasi-isomorophisms and $l^*$ denotes the dual map to $l$. 
By considering a Sullivan model 
$\xymatrix@=15pt{ C^*(\mathbb{L}_{W\oplus \Q\{w_{\alpha}\}}) \ \ar@{>->}[r]  & D}$ 
for $l^*$ and applying Lifting lemma \cite[Lemma 3.6]{F-H-T1}, 
we have a commutative diagaram 
$$
\xymatrix@C25pt@R15pt{ 
A_{PL}(X) & A_{PL}(X\cup_{\alpha}e^{k+1}) \ar[l]_(0.6){A_{PL}(i)} \\
D \ar[u] \ar[d]&  C^*(\mathbb{L}_{W\oplus \Q\{w_{\alpha}\}})   
\ar[l] \ar[u] \ar[d]^{=}\\
C^*(\mathbb{L}_{W}) & 
C^*(\mathbb{L}_{W\oplus \Q\{w_{\alpha}\}}) \ar[l]_{l^*} 
}
$$
in which vertial arrows are quasi-isomorophisms. 
Thus from the naturality of the model 
due to Brown and Szczarba, we can construct a commutative diagram 
$$
\xymatrix@=20pt{
\Q[V] \ar[d]_{m_1}
  \ar[r]_(0.5){\simeq} & \bullet \ar[d] &  
  \bullet \ar[l]^{\simeq} \ar@<-2.3ex>[d] \cdots  \bullet 
    \ar[r]_{\simeq} \ar@<2.3ex>[d] & \bullet \ar[d] & 
A_{PL}(Y) \ar[l]^(0.6){\simeq} \ar[d]^{A_{PL}(ev_*)}   \\
\Q[V\otimes C_*(\mathbb{L}_{W})] \ar[d]_{\eta}
  \ar[r]_(0.6){\simeq} & \bullet \ar[d] &  
  \bullet \ar[l]^{\simeq}  \ar@<-2.3ex>[d]   \cdots   \bullet 
    \ar[r]_{\simeq}  \ar@<2.3ex>[d] & \bullet \ar[d] & 
A_{PL}(\F(X, Y)) \ar[l]^(0.6){\simeq} \ar[d]^{A_{PL}(i^\sharp)}   \\
\Q[V\otimes C_*(\mathbb{L}_{W\oplus \Q\{w_{\alpha}\}})]
  \ar[r]_(0.8){\simeq} & \bullet  &  
  \bullet \ar[l]^{\simeq}    \cdots   \bullet 
    \ar[r]_{\simeq} & \bullet & 
A_{PL}(\F(X\cup_{\alpha}e^{k+1}, Y)) \ar[l]^(0.8){\simeq}   
}
$$ 
in the category of DGA's in which all the  horizontal arrows are
quasi-isomorph\-isms (for the DGA's represented by dots, see \cite{B-S} 
and also \cite[Section 3]{K}, 
the previous and ensuring discussions). The reslts \cite[Proposition 5.3]{K} 
and \fullref{lem:sullivan-model} assert that $m_1$ and $\eta$ 
are the inclusions of 
relative Sullivan algebras. Thus by applying Lifting lemma repeatedly, 
we have the two front commutative squares 
in \fullref{lem:model-map}.  The commutativity of the back
  square follows from that of the two side triangles. This completes the
  proof.  
\endproof

\bibliographystyle{gtart} \bibliography{link}

\end{document}